\documentclass[12pt]{amsart}
\usepackage[utf8]{inputenc}
\usepackage{amsmath}
\usepackage{amssymb}
\usepackage{amsthm}
\usepackage[margin=1.1in]{geometry}
\usepackage{csquotes}
\usepackage{xcolor}
\usepackage{hyperref}
\usepackage{enumerate}
\usepackage{comment}

\newtheorem{theorem}{Theorem}[section]

\newtheorem{lemma}[theorem]{Lemma}
\newtheorem{corollary}[theorem]{Corollary}
\newtheorem{definition}{Definition}
\newtheorem{proposition}[theorem]{Proposition}

\newtheorem{example}[theorem]{Example}

\newcommand{\F}{{\mathbb F}}

\newcommand{\Z}{{\mathbb Z}}

\newcommand{\GL}{\hbox{{\rm GL}}}
\newcommand{\K}{{\mathbb K}}

\newcommand{\fq}{{\mathbb F}_{q}}

\newcommand{\PG}{\mathrm{PG}}

\newcommand{\ZZZ}{\mathcal{Z}}

\newcommand{\rad}{\mathrm{rad}_q}

\usepackage{setspace}
\title{Prime Power Residues and Blocking Sets}
\author[B. Mishra]{Bhawesh Mishra} \author[P. Santonastaso]{Paolo Santonastaso}
\date{}

\begin{document}

\begin{abstract}
Let $q$ be a fixed odd prime. We show that a finite subset $B$ of integers, not containing any perfect $q^{th}$ power, contains a $q^{th}$ power modulo almost every prime if and only if $B$ corresponds to a blocking set (with respect to hyperplanes) in $\PG(\F_{q}^{k})$. Here, $k$ is the number of distinct prime divisors of $q$-free parts of elements of $B$. As a consequence, the property of a subset $B$ to contain $q^{th}$ power modulo almost every prime $p$ is invariant under geometric $q$-equivalence defined by an element of the projective general linear group $\mathrm{PGL}(\mathbb{F}_{q}^{k})$. Employing this connection between two disparate branches of mathematics, Galois geometry and number theory, we classify, and provide bounds on the sizes of, minimal such sets $B$. 
\end{abstract}

\maketitle


\section{Introduction}
In this article, we are primarily concerned with finite subsets $B$ of integers that contain a $q^{th}$ power modulo almost every prime $p$. Here, the phrase \textit{almost every prime} will denote all but finitely many primes and \textit{perfect $q^{th}$ power} will mean $q^{th}$ power of an integer. Any $B$ that already contains a perfect $q^{th}$ power is trivially such a set. The more interesting case is when $B$ does not contain any perfect $q^{th}$ power but still contains a $q^{th}$ power modulo almost every prime. We will call such sets \textit{locally $q^{th}$ power sets}. 

Our starting point is the following classical result on locally $2^{nd}$ power sets, which was first obtained by Fried in \cite{fried1969arithmetical}, and later rediscovered by Filaseta and Richman in \cite{filaseta1989sets}.

\begin{proposition}
Let $a_{1}, a_{2}, \ldots , a_{l}$ be finitely many nonzero integers. Then the following three conditions are equivalent:
\begin{enumerate}
    
    \item The set $\{a_{1}, a_{2}, \ldots, a_{l}\}$ contains a square modulo almost every prime.\vspace{1mm} 
    
    \item For every odd prime $p \nmid \prod_{j=1}^{l} a_{j}$, the set $\{a_{1}, a_{2}, \ldots, a_{l}\}$ contains a square modulo $p$.\vspace{1mm}
    
    \item There exists $T \subseteq \{1, 2, \ldots , l\}$ of odd cardinality such that $\prod_{j\in T} a_{j} $ is a perfect square.
    
\end{enumerate}
\label{FR}
\end{proposition}

The analogous result for general powers was obtained by A. Schinzel and M. Skałba in \cite{SS} and is combinatorially quite complex in nature. The main result in \cite{SS} deals with a more general problem and considers subsets of rings of integers. We refer the readers to \cite[Theorems 1, 2]{SS} for the results obtained by Schinzel and Skałba. For prime powers, M. Skałba further simplified the result to obtain the following \cite{skalba2005sets}.

\begin{proposition}
Let $q$ be a prime and $B = \{ b_{1}, b_{2}, \ldots, b_{l}\}$ be a set of finitely many distinct non-zero integers. Then the following conditions are equivalent:
\begin{enumerate}

    \item The set $B$ contains a $q^{th}$ power modulo almost every prime. \vspace{1mm}
    
    \item For every prime $p \neq q$ and $p \nmid \prod_{j=1}^{l} b_{j}$, the set $B$ contains a $q^{th}$ power modulo $p$. \vspace{1mm}

    \item For each sequence of integers $\{c_{j}\}_{j=1}^{l}$, there exists a sequence of integers $\{f_{j}\}_{j=1}^{l}$ such that \begin{equation}
\sum_{j=1}^{l} f_{j} \not\equiv 0 (\text{ mod } q) \text{ and } \prod_{j=1}^{l} b_{j}^{c_{j}f_{j}} = d^{q} \text{ for some integer } d. \label{skalbalemmaequation}
\end{equation}
\end{enumerate}
\label{Skalbalemma}
\end{proposition}

This article focuses on extremal sizes and multiplicative structure of locally $q^{th}$ power sets. We bridge number theory and finite geometry by providing a geometric characterization of locally $q^{\text{th}}$ power sets. Specifically, we show that locally $q^{th}$ power sets $B$ correspond to blocking sets (with respect to hyperplanes) in projective space $\mathrm{PG}(\mathbb{F}_{q}^{k})$. Here, $k$ is equal to the number of distinct primes dividing the $q$-free parts of elements of $B$ (see Section \ref{section:main}). This geometric characterization allows us to introduce the notion of \textit{geometric $q$-equivalence}, which, we will establish, is an equivalence relation between locally $q^{th}$ power sets. This allows us to construct and recognize locally $q^{th}$ power sets from existing ones (see Section \ref{section:geomeq}). Moreover, our methods enable us to study subtler details of locally $q^{th}$ power sets such as their dimensions and lower and upper bounds on their cardinalities. In some cases, classifications of minimum size locally $q^{th}$ power sets are also provided (see Sections \ref{section:classification}, \ref{section:2dim} and \ref{section:3dim}).

\section{Preliminaries}\label{section:prelim}
Given a prime $p$, an integer $ \nu\geq 1$ and $a \in \mathbb{Z}$, we write $p^{\nu} \mid\mid a$, if $p^{\nu}$
is the highest power of $p$ dividing $a$. In this article, $q$ will always denote a fixed odd prime. 
\subsection{Reduction to the positive q-free equivalence class}
Let $B =\{b_{1}, b_{2}, \ldots, b_{\ell}\}$ be a finite subset of integers. For any prime $p$, an integer $b$ is a $q^{th}$ power modulo $p$ if and only if $|b|$ is a $q^{th}$ power modulo $p$. This is because $-1$ is always a perfect $q^{th}$ power. Therefore, we can replace the set $B$ by the set $\{|b_{j}|\}_{j=1}^{\ell}$.

Given a positive integer $b$ with unique prime factorization $\prod_{j=1}^{k} p_{j}^{a_{j}}$ $(a_{j}\geq 1)$, its \emph{$q$-free part} is denoted by $\text{rad}_{q}(b)$, and can be defined as \[\text{rad}_{q}(b) := \prod_{j=1}^{k} p_{j}^{a_{j} \text{mod } (q)}.\] For a prime $p \neq q$ and $p \nmid b$, $b$ is a $q^{th}$ power modulo $p$ if and only if $\text{rad}_{q}(b)$ is a $q^{th}$ power modulo $p$. Since there are only finitely many primes dividing $p$ that divide $\prod_{b\in B} b$, we can see that a set $B = \{b_{j}\}_{j=1}^{\ell} \subset\mathbb{Z}$ contains a $q^{th}$ power modulo almost every prime if and only if the set $\{\text{rad}_{q}(\lvert b_{j} \rvert)\}_{j=1}^{\ell}$ contains a $q^{th}$ power modulo almost every prime. Therefore, we introduce an equivalence relation $\sim_{q}$ on the set $\mathbb{Z}\setminus\{0\}$ defined as follows:

Let $b_{1}, b_{2}$ be two non-zero integers and $p_{1}, p_{2}, \ldots, p_{k}$ be all the primes dividing $b_{1}b_{2}$. Let $b_{1} = \pm \prod_{i=1}^{k} p_{i}^{\mu_{i}}$ $(\mu_{i} \geq 0)$ and $b_{2} = \pm \prod_{i=1}^{k} p_{i}^{\nu_{i}}$ $(\nu_{i} \geq 0)$. We will say that 
\begin{equation*}
    b_{1} \sim_{q} b_{2} \text{ if and only if } \mu_{i} \equiv \nu_{i} \hspace{1mm} (\text{mod } q) \text{ for every } 1 \leq i \leq k.
\end{equation*}
$\sim_{q}$ is clearly an equivalence relation, where the equivalence class containing an integer $b$ will be denoted by $[b]_{q}$. Let $\mathcal{Z}$ be the set of equivalence classes and \[\pi_{q}: \mathbb{Z}\setminus\{0\} \rightarrow \mathcal{Z}\] be the canonical map that takes an integer $b$ to its equivalence class $[b]_{q}$. 
Note that $b \sim_{q} \text{rad}_{q}(|b|)$ for every integer $b$, and therefore we will often identify $[b]_{q}$ with $\text{rad}_{q}(|b|)$. 

Therefore, to study whether a given finite subset $B = \{b_{j}\}_{j=1}^{\ell} \subset\mathbb{Z}\setminus\{0\}$ is a locally $q^{th}$ power set, it suffices to investigate whether $\pi_{q}(B) = \{\text{rad}_{q}(|b_{j}|)\}_{j=1}^{\ell}\subset\mathcal{Z}$ is a locally $q^{th}$ power set. Moreover, if $B$ is a locally $q^{th}$ power set, then every set $A$, with $A \supset B$, is also a locally $q^{th}$ power set. Keeping this in mind, we introduce the following definition:

\begin{definition}
An element $b$ of a locally $q^{th}$ power set $B$ is said to be \textbf{essential} if the set $B\setminus\{b\}$ is no longer a locally $q^{th}$ power set. We will say that a locally $q^{th}$ power set $B$ is minimal if every element of $B$ is essential. 
\end{definition}

Clearly, if $B$ has two distinct elements $c$ and $d$ such that $c \sim_{q} d$, then $B$ cannot be minimal. Therefore, we have the following immediate assertion.
\begin{lemma} \label{lem:sizeqfree}
If a finite subset $B$ of non-zero integers is a minimal locally $q^{th}$ power set, then $|B| = |\pi_{q}(B)|$. \label{basicminimal}
\end{lemma}

Following is an example of a minimal locally $3^{rd}$ power set. 

\begin{example}
Let $q = 3$, $p_{1}, p_{2}$ be two distinct primes and consider the set \[ B = \{ p_{1}, p_{2}, p_{1}p_{2}, p_{1}p_{2}^{2} \}. \] One can easily see that the set $B$ is a locally $3^{rd}$ power set. If neither $p_{1}$ nor $p_{2}$ is cube modulo $p$ then, two cases arise in $G = \mathbb{F}_{p}^{\times}/\mathbb{F}_{p}^{\times 3}$:
\begin{enumerate}
    \item Both $p_{1}$ and $p_{2}$ are the same non-identity element in $G$. In this case, $p_{1}p_{2}^{2}$ is a cube modulo $p$.

    \item If not, $p_{1}$ and $p_{2}$ are distinct in $G$, in which case, $p_{1}p_{2}$ is a cube modulo $p$. 
\end{enumerate}
In either case, $B$ is a locally $3^{rd}$ power set since it contains a cube modulo $p$ for every prime $p\not\in\{3, p_{1}, p_{2}\}$. Since every locally $3^{rd}$ power set has at least four elements (for instance, see \cite[Corollary 1, pp.10]{mishra2023prime}), it also follows that $B$ is minimal locally $3^{rd}$ power set. In fact, we will establish in this article that all minimal locally $q^{th}$ power sets are of the form \[ C := \{ p_{1}, p_{2}, p_{1}p_{2}, p_{1}p_{2}^{2}, \ldots, p_{1}p_{2}^{q-1}\} \] up to a suitable equivalence. The equivalence is needed since, for instance, any element in $C$ could be modified by a perfect $q^{th}$ power multiple and the resulting set would still remain a locally $q^{th}$ power set. 
\end{example}

\subsection{Blocking sets}
Let $\K$ be a field and $V$ be a vector space over $\K$. Let $\PG(V,\K)$ be the projective space defined over $\K$ by $V$. In this article, $\K$ will always be the Galois field $\mathbb{F}_{q}$ of order $q$. Therefore, we will abbreviate $\PG(V, \mathbb{F}_{q})$ as $\PG(V)$.
\begin{definition}
Let $k \geq 2$. A \textbf{blocking set} $\mathcal{S}$ (with respect to hyperplanes) of $\PG(\mathbb{F}_{q}^{k})$ is a collection of points that meets every hyperplane in at least one point. \label{defnblocking} 
\end{definition}

Any line in $\PG(\mathbb{F}_{q}^{k})$ is a blocking set, and any blocking set that contains a line will be called a \textit{trivial} blocking set. Similar to the case of locally $q^{th}$ power set, we define minimal blocking sets as follows.

\begin{definition}
A point $\mathcal{P}$ of a blocking set $\mathcal{S}$ of $\PG(\F_q^k)$ is said to be \textit{essential} if $\mathcal{S} \setminus \{\mathcal{P}\}$ is no longer a blocking set,
i.e. there is a hyperplane of $\PG(\F_q^k)$ meeting $\mathcal{S}$ in $\mathcal{P}$ only. Hence a blocking set is \textit{minimal} if and only if every point is essential. Equivalently a blocking set is minimal if it does not properly contain a blocking set.  
\end{definition}
Readers can also consult \cite{blokhuis2011blocking} for an excellent survey on the topic.
We will say that a blocking set $\mathcal{S}$ of $\PG(\F_{q}^{k})$ is $r$-dimensional if the projective subspace of $\PG(\F_{q}^{k})$ generated by $\mathcal{S}$ has dimension $r$, In the case $r=2$, we say that $\mathcal{S}$ is planar, i.e., if the projective subspace generated by $\mathcal{S}$ is a plane of $\PG(\F_{q}^{k})$.
We note that if $W$ is a subspace of $\F_{q}^{k}$, then any blocking set of $\PG(W)$ is also a blocking set of $\PG(\F_{q}^{k})$.

\section{Blocking Sets and Locally $q^{th}$-Power Sets} \label{section:main}
Before we can characterize locally $q^{th}$ power sets in terms of blocking sets, we will need a few notation. 

\subsection{Point set associated to a finite subset of integers}\label{pointset}
Let $B = \{b_{j}\}_{j=1}^{\ell}$ be a finite subset of non-zero integers not containing any perfect $q^{th}$ power. Let $p_{1} < p_{2} < \ldots < p_{k}$ be all the distinct primes that divide any element of $\pi_{q}(B)$. For every $1 \leq i \leq k$ and for every $1 \leq j \leq \ell$, let $\nu_{ij}\geq 0$ be such that $p_{i}^{\nu_{ij}}\mid\mid \text{rad}_{q}(\lvert b_{j}\rvert)$. For every $1 \leq j \leq \ell$, let $\mathcal{P}_{j} = \langle (\nu_{1j}, \nu_{2j}, \ldots, \nu_{kj})\rangle \in \PG(\F_{q}^{k})$. The set $\{\mathcal{P}_{j}\}_{j=1}^{\ell} \subset\PG(\F_{q}^{k})$ will be called the point set (in $\PG(\F_{q}^{k})$) associated to $B \subset\mathbb{Z}\setminus\{0\}$. Our first result is the following:  

\begin{theorem}
Let $B$ be a finite subset of non-zero integers not containing a perfect $q^{th}$ power and let $k$ be the total number of primes dividing $\prod_{b\in B} \text{rad}_{q}(\lvert b_{j} \rvert)$. Then, the following two statements are equivalent:
\begin{enumerate}
        \item The set $B$ is a locally $q^{th}$ power set. 
        
        \item The point set $\{\mathcal{P}_{j}\}_{j=1}^{\ell}$, associated to $B$, forms a blocking set in $\PG(\F_{q}^{k})$. 
    \end{enumerate}
    \label{mainresult}
\end{theorem}
In order to prove Theorem \ref{mainresult}, we will use a characterization of locally $q^{th}$ power sets in terms of linear hyperplane coverings of vector spaces, obtained by the first author in \cite{mishra2023prime}.

Let $V$ be a vector space of dimension $\geq 2$ over a field $\mathbb{K}$. A collection $\{W_{i}\}_{i\in I}$ of proper $\mathbb{K}$-subspaces is said to be a \textit{linear covering} if $V = \bigcup_{i \in I} W_{i}$.

\subsection{Hyperplanes Associated with a Finite Subset}\label{hyperplanes}
Let $B = \{b_{j}\}_{j=1}^{\ell}$ be a finite subset of non-zero integers. Let $p_{1}, p_{2}, \ldots, p_{k}$ be all the primes that divide any element of $\pi_{q}(B)$. For every $1 \leq i \leq k$ and for every $1 \leq j \leq \ell$, let $\nu_{ij}\geq 0$ be such that $p_{i}^{\nu_{ij}}\mid\mid \text{rad}_{q}(b_{j})$. For every $1 \leq j \leq \ell$, define the $\mathcal{H}_{j}$ as, 
\[\mathcal{H}_{j} := \big\{ (x_{i})_{i=1}^{k} \in\mathbb{F}_{q}^{k} : \sum_{i=1}^{k} \nu_{ij} x_{i} = 0 \big\}.\] We will say that $\{\mathcal{H}_{j}\}_{j=1}^{\ell}$ is the set of hyperplanes (in $\mathbb{F}_{q}^{k}$) associated with the set $B$. The Proposition \ref{Skalbalemma} was further simplified in terms of linear covering in \cite{mishra2023prime}. 

\begin{proposition}  \label{th:connectioncoveringqpowerset} 
Let $q$ be an odd prime, $B = \{b_{j}\}_{j=1}^{\ell}$ be a finite subset of non-zero integers that does not contain a perfect $q^{th}$ power. Then, the following two statements are equivalent. 
\begin{enumerate}
    \item $B$ is a locally $q^{th}$ power set. 

    \item \[\bigcup_{j=1}^{\ell} \PG\big(\mathcal{H}_{j}\big) = \PG(\mathbb{F}_{q}^{k}\big) \Bigg( \text{ or equivalently } \bigcup_{j=1}^{\ell} \mathcal{H}_{j} = \F_{q}^{k}, \Bigg)\] where $\{\mathcal{H}_{j}\}_{j=1}^{\ell}$ is the set of hyperplanes associated with $B$.
\end{enumerate}
\end{proposition}

\subsection{Proof of Theorem \ref{mainresult}}

From now on, we consider the projective space $\PG(\F_q^k)$ together with a coordinatization $(x_1,\ldots,x_k)$. We also consider the following duality $\perp$ on $\PG(\F_q^k)$. To any projective subspace $\Omega=\PG(W)$, we associate $\Omega^{\perp}=\PG(W^{\perp})$, where \[W^{\perp}=\{(u_1,\ldots,u_k) \in \fq^k \colon \sum_{i=1}^k u_iw_i=0, \mbox{for each }(w_1,\ldots,w_k) \in W\}.\]In particular, to a hyperplane $\mathcal{H}$ of $\PG(\F_q^k)$ defined by the equation $v_1X_1+\cdots+v_kX_k=0$ corresponds the point $\mathcal{H}^{\perp}=\langle (v_1,\ldots,v_k) \rangle_{\fq}$, and conversely. The duality $\perp$ gives is a 1-to-1 correspondence between points and hyperplanes of $\PG(\F_q^k)$ and preserves the incidence, i.e. if $\mathcal{P}$ is a point and $\mathcal{H}$ is a hyperplane of $\PG(\F_q^k)$ then $\mathcal{P} \in \mathcal{H}$ if and only if $\mathcal{H}^{\perp} \in \mathcal{P}^{\perp}$. The dual of a linear covering, with hyperplanes, is a blocking set. We include a statement and its proof for the sake of completeness. 

\begin{lemma}
Let $\ell \in\mathbb{N}$ and for every $1 \leq j \leq \ell$, let $\mathcal{H}_{j}$ be the hyperplane of $\PG(\F_{q}^{k})$ defined as \[\big\{(x_{i})_{i=1}^{k} : \sum_{i=1}^{k} \nu_{ij}x_{i} = 0 \big\}.\] Then, the following two statements are equivalent:
\begin{enumerate}
    \item  $\bigcup_{j=1}^{\ell} \mathcal{H}_{j} = \PG(\F_{q}^{k})$.\\
    \item The set $\mathcal{S} = \{\mathcal{H}_{j}^{\perp}\}_{j=1}^{\ell}$ forms a blocking set in $\PG(\F_{q}^{k})$.   
\end{enumerate}
\label{coveringtoblocking}
\end{lemma}

\begin{proof}
Let $\mathcal{H}$ be a hyperplane of $\PG(\F_{q}^{k})$. We need to show that $\mathcal{H}$ intersects the set $\mathcal{S}$. Since $\mathcal{P}=\mathcal{H}^{\perp} \in\PG(\F_{q}^{k}) = \bigcup_{j=1}^{\ell} \mathcal{H}_{j}$, we have that $\mathcal{P} \in \mathcal{H}_{j_{0}}$ for some $1 \leq j_{0} \leq \ell$. Since $\perp$ preserves the incidences between points and hyperplanes, $\mathcal{P}\in\mathcal{H}_{j_{0}}$ implies $\mathcal{H}_{j_{0}}^{\perp} \in \mathcal{P}^{\perp} = \mathcal{H}$, i.e., $\mathcal{H}_{j_{0}} \in \big(\mathcal{H} \cap\mathcal{S}\big)$. 

For the proof of the other direction, let $\mathcal{P}$ be a point in $\PG(\F_{q}^{k})$. We want to show that $\mathcal{P}\in\mathcal{H}_{j_{0}}$ for some $1 \leq j_{0} \leq \ell$. Since $\mathcal{S}$ is a blocking set, the hyperplane $\mathcal{H} = P^{\perp}$ in $\PG(\F_{q}^{k})$ must intersect $\mathcal{S}$. In other words, there exists $1 \leq j_{0} \leq \ell$ such that $H_{j_{0}}^{\perp} \in \mathcal{H} = \mathcal{P}^{\perp}$, which gives us $\mathcal{P} \in \mathcal{H}_{j_{0}}$. 
\end{proof}

Since every locally $q^{th}$ power set is associated to a linear covering of $\PG(\F_{q}^{k})$, Theorem \ref{mainresult} is a consequence of Proposition \ref{th:connectioncoveringqpowerset} and the Lemma \ref{coveringtoblocking}.

Theorem \ref{mainresult} enables us to employ various tools and results from the field of Galois geometry to extract number-theoretic information regarding locally $q^{th}$ power sets. To wrap up this section, we will present few corollaries of Theorem \ref{mainresult}. The first corollary below demonstrates that the concept of $(i)$ minimal locally $q^{th}$ power set and $(ii)$ minimal blocking set, correspond to one another as expected. 

\begin{corollary}
Using the notations in Theorem \ref{mainresult}, we have that $B$ is a minimal locally $q^{th}$ power set if and only if the set $\mathcal{S}$ of points in $\PG(\F_{q}^{k})$ forms a minimal blocking set of $\PG(\F_{q}^{k})$ and $|\mathcal{S}| = |B|$. \label{minimalblockingandminimalpowerset}
\end{corollary}

\begin{proof}
Note that every element $\mathcal{H}_{j}^{\perp}$ of $\mathcal{S}$ is associated to an element of $B$. This corollary is a straightforward consequence of Lemma \ref{lem:sizeqfree} and the fact that the following two statements are equivalent, as a consequence of Theorem \ref{mainresult}:
\begin{itemize}
    \item The set $\mathcal{S}\setminus\{\mathcal{H}_{j_{0}}^{\perp}\}$ is a blocking set of $\PG(\F_{q}^{k})$. 

    \item The set $B \setminus\{b_{j_{0}}\}$ is a locally $q^{th}$ power set,
\end{itemize}
see Proposition \ref{th:connectioncoveringqpowerset}.
\end{proof}

Another consequence of Theorem \ref{mainresult} is the following lower bound on non-trivial locally $q^{th}$ power sets. 

\begin{proposition} [see \textnormal{\cite[Theorem 2.1]{blokhuis2011blocking} and \cite{bose1966characterization}}] 
Let $k \geq 2$. Any blocking set $B$ of $\PG(\F_{q}^{k})$ has at least $q+1$ points. In case of equality the blocking set is the point set of a line. \label{th:firstboundblocking}
\end{proposition}

As a consequence, in unison with Theorem \ref{mainresult}, we obtain the following corollary that was also proved in \cite{mishra2023prime}. 

\begin{corollary} \label{prop:leqq}
A subset $B$ of integers, with cardinality at most $q$, contains a $q^{th}$ power modulo almost every prime if and only if $B$ contains a perfect $q^{th}$ power. In particular, any locally $q^{th}$ power set has size at least $q+1$.
\end{corollary}

\section{Geometric $q$-Equivalence} \label{section:geomeq}
To start the discussion pertaining to this section, consider the following proposition that shows that the property of a finite subset of integers to be a locally $q^{th}$ power set is invariant under exponentiation by elements of $\F_{q}\setminus\{0\}$ in the following sense. 
\begin{proposition} [see \textnormal{\cite[Corollary 2]{mishra2023prime}}]
Let $B = \{b_{j}\}_{j=1}^{\ell}$ be a finite subset of integers. Given $\Vec{c}=(c_{j})_{j=1}^{\ell} \in\Big(\F_{q}\setminus\{0\}\Big)^{l}$, define $B^{\Vec{c}} = \{b_{j}^{c_{j}}\}_{j=1}^{\ell}$. $B$ is a locally $q^{th}$ power set if and only if $B^{\Vec{c}}$ is a locally $q^{th}$ power set for every $\Vec{c}\in\F_{q}\setminus\{0\}$. \label{exponentiation} 
\end{proposition}

Element-wise exponentiation in Proposition \ref{exponentiation} is a particular instance of a much more general set of transformations, that we will call \textit{geometric $q$-equivalence}. In this section, we will explain geometric $q$-equivalence and show that the property of being a locally $q^{th}$ power set is invariant under it. We will denote the projective general linear group of degree $k$ over $\F_{q}$ by $\mathrm{PGL}(k,q)$.

\begin{definition}
Let $B = \{a_{j}\}_{j=1}^{m}$ and $B^{\prime} = \{b_{j}\}_{j=1}^{\ell}$ be two finite subsets of non-zero integers, not containing a perfect $q^{th}$ power. Let $p_{1}, p_{2}, \ldots, p_{k}$ be all the primes that divide any element of $\pi_{q}(B) \cup \pi_{q}(B^{\prime})$. Let $\text{rad}_{q}(|a_{j}|) = \prod_{i=1}^{k} p_{i}^{\nu_{ij}}$
for every $j\in\{1, 2, \ldots, m\}$ and $\text{rad}_{q}(|b_{j}|) = \prod_{i=1}^{k} p_{i}^{\mu_{ij}}$ for every  $j \in\{1, 2, \ldots, \ell\}$, where $\nu_{ij}, \mu_{ij} \geq 0$. Define 
\begin{multline*}
\\ \mathcal{S} = \big\{ \langle ( \nu_{1j}, \nu_{2j}, \ldots, \nu_{kj} )\rangle_{\fq} \in\PG(\F_{q}^{k}) : j\in\{1, 2, \ldots, m\} \big\} \text{ and }\\ \mathcal{S}^{\prime} = \big\{\langle (\mu_{1j}, \mu_{2j}, \ldots, \mu_{kj} ) \rangle_{\fq} \in\PG(\F_{q}^{k}) : j\in\{1, 2, \ldots, \ell\} \big\} \\
\end{multline*}
to be the point sets in $\PG(\mathbb{F}_{q}^{k})$ associated to $B$ and $B^{\prime}$ respectively. We will say that the sets $B$ and $B^{\prime}$ are \textit{geometrically q-equivalent} if and only if there exists an element $\Psi \in\mathrm{PGL}(k,q)$ such that $\Psi(\mathcal{S}) = \mathcal{S}^{\prime}$.
\end{definition}

Our next result shows that property whether a given finite $B\subset\mathbb{Z}\setminus\{0\}$, is a locally $q^{th}$ power set is preserved by geometric $q$-equivalence.  
\begin{theorem} \label{th:propertygeometricequiv}
    Let $B=\{b_1,b_2,\ldots,b_t\} \subset \Z \setminus \{0\}$ be a set of integers not containing a perfect $q^{th}$ power. Then $B$ is a locally $q^{th}$ power set if and only if every set $B^{\prime}$ that is geometrically $q$-equivalent to $B$ is a locally $q^{th}$ power set, as well. 
\end{theorem}

\begin{proof}
    We will show that if $B$ is a locally $q^{th}$ power set, then every set $B^{\prime}$ that is geometrically $q$-equivalent to $B$ is also a locally $q^{th}$ power set. The converse holds trivially.
    
    Let $ \{p_1,\ldots,p_s\}$ be the set of all prime numbers dividing the elements of \[\pi_q(B)=\{ 
    \rad( b_1 ),\ldots,\rad( b_t )
    \} \] such that $p_i^{v_{ij}} \mid \mid \rad( b_j )$, for every $1\leq i \leq s$ and every $1 \leq j \leq t$. By Lemma \ref{coveringtoblocking}, the set of points $\mathcal{S}=\{\mathcal{P}_1,\ldots,\mathcal{P}_t\} \subseteq \PG(\F_{q}^{s})$, where 
        \[
        \mathcal{P}_j=\langle (v_{1j},\ldots,v_{sj}) \rangle_{\F_q}
        ,\]
        for each $1 \leq j \leq t$, is a blocking set of $\PG(\F_{q}^{s})$. First, assume that $B'$ is a geometrically $q$-equivalent set to $B$ and let $\pi_{q}(B^{\prime}) = \{a_1, \ldots, a_l\}$. Let $p_{s+1}, \ldots, p_{k}$, with $k \geq s$ be some primes such that $\{p_1,\ldots,p_s,p_{s+1},\ldots,p_k\}$ is the set of all the primes that divides the elements of $B \cup B'$. Suppose that $\lvert a_j \rvert=\prod_{i=1}^k p_i^{w_{ij}}$, for $j \in \{1,\ldots,l\}$ and $\lvert b_h \rvert=\prod_{i=1}^kp_i^{v_{ih}}$, for $h \in \{1,\ldots,t\}$ (Note that $v_{ij}=0$, if $i \geq s+1$). Since, $B'$ is geometrically $q$-equivalent to $B$ then there exists an element $\Psi \in \mathrm{PGL}(k,q)$, such that 
    \[\Psi(\mathcal{S})=\mathcal{S}',\]
    where $\mathcal{S}=\{\langle (w_{1j},\ldots,w_{kj}) \rangle_{\F_q} \colon j \in \{1,\ldots,l\}\}$ and  $\mathcal{S}'=\{\langle (v_{1h},\ldots,v_{kh}) \rangle_{\F_q} \colon h \in \{1,\ldots,t\}\}$. Since the $\mathcal{P}_j$'s form a blocking set of $\PG(\F_{q}^{s})$, they define also a blocking set in $\PG(\F_{q}^{k})$, because $k \geq s$, and hence $\mathcal{S}$ is a blocking set in $\PG(\F_{q}^{k})$. The property of being a blocking set is invariant under the action of $\mathrm{PGL}(k,q)$ on $\PG(\F_q^k)$, and then we get that also $\mathcal{S}'$ is a blocking set. Again by Theorem \ref{coveringtoblocking}, we have that $B'$ is a $q^{th}$-power set, and then the assertion follows. 
\end{proof}
One benefit of Theorem \ref{th:propertygeometricequiv} is that one can systematically study locally $q^{th}$ power sets according to the number of distinct primes that divide them. Assume that the point set of $\PG(\F_{q}^{k})$ associated to the set $B \subset\mathbb{Z}\setminus\{0\}$ lie in a proper subspace of $\PG(\F_{q}^{k})$. In this case,  $B$ is geometrically $q$-equivalent to a set $B^{\prime}$ that is divisible by fewer primes. We will demonstrate this first using an illustrative example.  

\begin{example}
    Choose $q=7$. Let \[B=\{2,15,30,60,120,240,480,960\}=\{2,3\cdot 5,2\cdot 3 \cdot 5,2^2\cdot 3 \cdot 5,2^3\cdot 3 \cdot 5,2^4\cdot 3 \cdot 5,2^5\cdot 3 \cdot 5, 2^6\cdot 3 \cdot 5\}\]
    The set of all primes that divide an element of $B$ is $\{2, 3, 5\}$.  The set $B$ defines the points 
\begin{multline*}
    \mathcal{P} = \Bigg\{\langle (1,0,0) \rangle_{\F_q},\langle (0,1,1) \rangle_{\F_q},\langle (1,1,1) \rangle_{\F_q},\langle (2,1,1) \rangle_{\F_q},\\ \langle (3,1,1) \rangle_{\F_q},\langle (4,1,1)  \rangle_{\F_q}, \langle (5,1,1) \rangle_{\F_q},\langle (6,1,1) \rangle_{\F_q} \Bigg\} 
\end{multline*}
    in $\PG(\F_{7}^{3})$, which is the line of $\PG(\F_{7}^{3})$ through the points $\langle (1,0,0) \rangle_{\F_q}$ and $\langle (0,1,1) \rangle_{\F_q}$. Consider
    \[
    A=\begin{pmatrix}
        1 & 0 & 0 \\
        0 & 1 & 0 \\
        0 & 1 & 6
    \end{pmatrix} \in \mathrm{PGL}(3,7).
    \] The set $\mathcal{P}$ is mapped to the set
    \begin{multline*}
    \mathcal{P}^{\prime} = \{ \langle (1,0,0)\rangle_{\fq},\langle(0,1,0)\rangle_{\fq},\langle(1,1,0)\rangle_{\fq},\langle(2,1,0)\rangle_{\fq}, \\\langle(3,1,0)\rangle_{\fq},\langle(4,1,0)\rangle_{\fq},\langle(5,1,0)\rangle_{\fq},\langle(6,1,0)\rangle_{\fq}\}\end{multline*} by the element of $\mathrm{PGL}(3,7)$ induced by matrix $A$. Note that $\mathcal{P}^{\prime}$ is associated to the set
    \[
   B^{\prime} = \{2, 3 ,2\cdot 3,2^2 \cdot 3,2^3 \cdot 3, 2^4 \cdot 3, 2^5\cdot 3,2^6 \cdot 3\}=\{2,3,6,12,24,48,96,192\}.
    \] that is divisible only by primes $2$ and $3$. To summarize, $B$ is geometrically $7$-equivalent to a set $B^{\prime}$, which is only divisible by $2$ primes. This is because $\mathcal{P}$ spans a proper subspace of $\PG(\F_{7}^{3})$. 
\end{example}
Consider the following definition. 
\begin{definition}
Let $B\subset\mathbb{Z}\setminus\{0\}$ be a finite set of integers and assume that there are exactly $k$ distinct primes dividing any element of $\pi_{q}(B)$. For a positive integer $d \leq k-1$, we will say that $B$ is $d$-dimensional if the point set in $\PG(\F_{q}^{k})$ associated to $B$ spans a $d$-dimensional subspace.
\end{definition}

Since $k$ counts the number of distinct primes dividing the elements of $\pi_{q}(B)$, one expects that if $B$ is $d$-dimensional, $B$ must ultimately be made up of $(d+1)$ primes. This is essentially the case, up to geometric $q$-equivalence, which is the content of the next proposition. 

\begin{proposition} \label{prop:reduceprimelength}
  Consider $B=\{b_1,b_2,\ldots,b_t\} \subset \Z \setminus \{0\}$ be a set of integers not containing a perfect $q^{th}$ power and assume that $B$ is $(r-1)$-dimensional. Then, for any choice of $r$ different primes $\overline{p}_1,\ldots,\overline{p}_r$, there exists a set $B^{\prime}$, that is geometrically $q$-equivalent to $B$, such that $\{\overline{p}_1,\ldots,\overline{p}_{r}\}$ is the set of all the primes numbers that divides at least one element of $B^{\prime}$.
\end{proposition}

\begin{proof}
Let $ \{p_1,\ldots,p_k\}$ be the set of all prime numbers diving the elements of \[\pi_q(B)=\{ 
    \rad( b_1 ),\ldots,\rad( b_t )
    \} \] such that $p_i^{v_{ij}} \mid \mid \rad( b_j )$, for every $1\leq i \leq k$ and every $1 \leq j \leq t$. 
 And consider the set of points $\mathcal{S}=\{\mathcal{P}_1,\ldots,\mathcal{P}_t\} \subseteq \PG(\F_{q}^{k})$, where 
        \[
        \mathcal{P}_j=\langle (v_{1j},\ldots,v_{kj}) \rangle_{\F_q}
        ,\]
        for each $1 \leq j \leq t$. Since $B$ is $(r-1)$-dimensional, the projective subspace $\langle \mathcal{S} \rangle$, generated by $\mathcal{S}$ has dimension $r-1$. Let $\Psi$ be the element of $\mathrm{PGL}(k,q)$ mapping $\langle \mathcal{S} \rangle$ in the $(r-1)$-dimensional subspace of $\PG(\F_{q}^{k})$ having equation $X_{r+1}=\ldots=X_k=0$. In other words, $\Psi(\mathcal{S})=\mathcal{S}'$, where $\mathcal{S}'=\{\mathcal{Q}_1,\ldots,\mathcal{Q}_t\} \subseteq \PG(\F_{q}^{k})$, where 
        \[
        \mathcal{Q}_j=\langle (w_{1j},\ldots,w_{kj}) \rangle_{\F_q}
        ,\]
        for each $1 \leq j \leq t$, with the properties that $w_{ij}=0$, if $i \geq r+1$. So the set \[
        B'=\left\{\prod_{i=1}^{r}p_i^{w_{i1}},\ldots,\prod_{i=1}^{r}p_i^{w_{it}}\right\}
        \]
        is geometrically $q$-equivalent to $B$, having the desired property.
\end{proof}

In certain cases, Proposition \ref{prop:reduceprimelength} can be employed to show that a given set is not a locally $q^{th}$ power set. Consider the following example: 
\begin{example}
    Choose $q=7$. Let \[B=\{2,3,4,5,6,10,15,25\}=\{2,3,2^2,5,2\cdot 3,2 \cdot 5,3 \cdot 5, 5^2\},\] which has $8$ elements. The set of all primes that divide an element of $B$ is $\{2, 3,5\}$. The point set
    \begin{multline*}
         \mathcal{S} = \Big\{\langle (1,0,0) \rangle_{\F_q},\langle (0,1,0) \rangle_{\F_q},\langle (2,0,0) \rangle_{\F_q},\langle (0,0,1) \rangle_{\F_q}, \\ \langle (1,1,0) \rangle_{\F_q},\langle (1,0,1) \rangle_{\F_q},\langle (0,1,1) \rangle_{\F_q},\langle (0,0,2) \rangle_{\F_q}\Big\}
    \end{multline*}
    in $\PG(\F_{7}^{3})$ is associated to $B$. However, as a point set in $\PG(\F_{7}^{3})$
    \[
    \mathcal{S} = \{\langle (1,0,0) \rangle_{\F_q},\langle (0,1,0) \rangle_{\F_q},\langle (0,0,1) \rangle_{\F_q},\langle (1,1,0) \rangle_{\F_q},\langle (1,0,1) \rangle_{\F_q},\langle (0,1,1) \rangle_{\F_q},\} 
    \]
    and hence, $B$ is geometrically $q$-equivalent to 
    \[B^{\prime} = \{2,3,5,2\cdot 3,2 \cdot 5,3 \cdot 5\}=\{2,3,5,6,10,15\}.\]
    Now the size of $B^{\prime}$ is $6 < 7$. Therefore, Proposition \ref{prop:leqq} implies that the set $B^{\prime}$ is not a locally $7^{th}$ power set, and hence using Theorem \ref{th:propertygeometricequiv}, $B$ is not a locally $7^{th}$ power set. 
\end{example}

\section{Classification and bounds of locally $q^{th}$ power set} \label{section:classification}
In this section, we will classify the locally $q^{th}$ power sets based on their cardinality and the shape of their elements. For instance, we describe the minimal locally $q^{th}$ power sets of smallest cardinality $q+1$. Next, we will describe the minimal locally $q^{th}$ power set of second smallest cardinality. Here, our starting point is the following well-known result. 

\begin{proposition}[see \cite{blokhuis1994size,heim1996blockierende}]
\label{th:secondminimumsizeblockingset} For $q$ an odd prime, let $\mathcal{S}$ be a non-trivial blocking set of $\PG(\F_{q}^{k})$, then 
$\lvert \mathcal{S} \rvert \geq \frac{3(q+1)}{2}$. In the case of equality the blocking set is planar.
\end{proposition}
We also prove that the bound $\lvert B \rvert \geq 3(q+1/2)$ is sharp, for any $q$ odd prime, completely classifying minimal locally $q^{th}$-power set in the case of the second smallest size $\lvert B \rvert=3(q+1)/2$ for $q=3,5$.

\begin{theorem} \label{th:classification}
    Let $B=\{b_1, \ldots, b_{\ell}\}$ be a locally $q^{th}$ power set and $\mathcal{S}$ be the point set in $\PG(\F_{q}^{k})$ associated to $B$. Here, $k$ is the number of distinct primes dividing $\prod_{j=1}^{\ell} \pi_{q}(b_{j})$. Then, $\lvert B \rvert \geq q+1$ and the following holds:
    \begin{enumerate}[i)]
        \item The cardinality of $B$ is $q+1$ if and only if $\mathcal{S}$ is a line of $\PG(\F_{q}^{k})$ and $B$ is geometrically $q$-equivalent to the set
        \begin{equation}
        \{\overline{p}_1, \overline{p}_2, \overline{p}_1\overline{p}_2, \overline{p}_1\overline{p}_2^2, \ldots, \overline{p}_1\overline{p}_2^{q-1}\}, \label{lineexample}
        \end{equation}
    for two distinct primes $\overline{p}_1,\overline{p}_2$. In this case, $B$ is $1$-dimensional and minimal. \vspace{2mm}
    
        \item Minimal locally $q^{th}$ power sets of cardinality strictly between $q+1$ and $\frac{3(q+1)}{2}$ do not exist. In particular, $q+1 < \lvert B \rvert < \frac{3(q+1)}{2}$ if and only if $B$ contains a proper subset that is geometrically $q$-equivalent to \eqref{lineexample}. \vspace{2mm}
        
        \item $B$ is minimal and $\lvert B \rvert=\frac{3(q+1)}{2}$ if and only if $B$ is $2$-dimensional, i.e., $\mathcal{S}$ is a blocking set of a plane $\PG(\F_{q}^{2})$ of $\PG(\F_{q}^{k})$. 
    \end{enumerate}
\end{theorem}

\begin{proof}
Since $B$ is a $q^{th}$ power set, Theorem \ref{mainresult} implies that $\mathcal{S}$ is a blocking set of $\PG(\F_{q}^{k})$. In this case, $\lvert B \rvert \geq q+1$ follows from Proposition \ref{th:firstboundblocking}. 

If $\lvert B \rvert = q+1$, then Proposition \ref{th:firstboundblocking} again implies that $\mathcal{S}$ is a line of $\PG(\F_{q}^{k})$. So $\mathcal{S}=\PG(1,q)$ and $B$ is $1$-dimensional. 

When $|B|$ lies strictly between $q+1$ and $\frac{3(q+1)}{2}$, and when $|B| = \frac{3(q+1)}{2}$ the result follows from Proposition \ref{th:secondminimumsizeblockingset}.
\end{proof}

Now we present two examples that demonstrate how Theorem \ref{th:classification} can be employed to obtain results about locally $q^{th}$ power sets. 
\begin{example}
Suppose we want to investigate whether the set $B=\{2,3,4,6,21,41\}$ is a locally $q^{th}$ power set for $q=5$. Note that $B=\{2,3,2^2,2 \cdot 3,3\cdot 7,2 \cdot 3 \cdot 7\}$ has size $6$. So by Theorem \ref{th:classification}, we have that $B$ is a $q^{th}$ power set if and only if the point set \[\mathcal{S} =\{ \langle (1,0,0)  \rangle_{\fq}, \langle (0,1,0)  \rangle_{\fq},\langle (2,0,0)  \rangle_{\fq},\langle (1,1,0)  \rangle_{\fq},\langle (0,1,1)  \rangle_{\fq},\langle (1,1,1)  \rangle_{\fq}\}\] defines a line of $\PG(\F_{5}^{3})$. However, the subspace generated by $\mathcal{S}$ must contain $\langle (1,0,0)  \rangle_{\fq},$ $\langle (0,1,0)  \rangle_{\fq}$ , $\langle (0,1,1)  \rangle_{\fq}$; so, $\mathcal{S}$ generates the whole space $\PG(\F_{5}^{3})$. In other words, $\mathcal{S}$ is not a line and $B$ is not a locally $5^{th}$ power set. 
\end{example}

\begin{example}
Similarly, $B=\{6,7,42,252,1512,18114\} =\{2 \cdot 3,7,2 \cdot 3 \cdot 7,2^2 \cdot 3^2 \cdot 7,2^3 \cdot 3^3 \cdot 7,2^4 \cdot 3^4 \cdot 7\}$ has size $6$. So by Theorem \ref{th:classification}, $B$ is a locally $5^{th}$ power set, because the associated point set \[\langle (1,1,0)  \rangle_{\fq}, \langle (0,0,1)  \rangle_{\fq},\langle (1,1,1)  \rangle_{\fq},\langle (2,2,1)  \rangle_{\fq}, \langle (3,3,1)  \rangle_{\fq},\langle (4,4,1)  \rangle_{\fq}\] defines the line of $\PG(\F_{q}^{3})$ through the points $ \langle (1,1,0)  \rangle_{\fq}$ and $\langle (0,0,1)  \rangle_{\fq}$. 
\end{example}

Note that Theorem \ref{th:classification} says that $1$-dimensional locally $q^{th}$ power set having size $q+1$ are associated to a line in the projective space. In the same spirit, in the next we investigate minimal $2$-dimensional locally $q^{\text{th}}$ power sets, which have cardinality $\frac{3(q+1)}{2}$. We provide constructions of such sets for every odd prime $q$. We will then show that, unlike the case of locally $q^{\text{th}}$ power sets of size $q+1$, here it is possible to give inequivalent constructions. Therefore, this shows that the structure of such sets depends on the value of $q$. Finally, we classify these sets for small values of $q$. We will introduce some preliminary definitions and facts before doing so.

\begin{definition}
A \textit{projective triangle} in $\PG(\F_{q}^{3})$, is a set $\mathcal{S}$ of $\frac{3(q+1)}{2}$ points such that:
\begin{enumerate}
    \item there exist three non collinear points $\mathcal{P}_1, \mathcal{P}_2,\mathcal{P}_3$ such that on each side of the triangle $\mathcal{P}_1 \mathcal{P}_2\mathcal{P}_3$ (i.e. $\mathcal{S} \cap \mathcal{P}_i\mathcal{P}_j$, with $i\neq j$), there are exactly $\frac{q-1}{2}+2$ points of $\mathcal{S}$;\vspace{2mm}
    
    \item the vertices $\mathcal{P}_1,\mathcal{P}_2,\mathcal{P}_3$ are in $\mathcal{S}$;\vspace{2mm}
    
    \item for $\{i,j,k\}=\{1,2,3\}$, the line connecting a point of $\mathcal{S} \cap \mathcal{P}_i \mathcal{P}_j$, with a point of $\mathcal{S} \cap \mathcal{P}_j \mathcal{P}_k$ also contain a point of $\mathcal{S} \cap \mathcal{P}_i \mathcal{P}_k$. 
\end{enumerate}
\end{definition}

A projective triangle is a minimal blocking set of $\PG(\F_{q}^{3})$ having size reaching the equality in Proposition \ref{th:secondminimumsizeblockingset}.

\begin{proposition} [see \textnormal{\cite[Lemma 13.6 (i)]{hirschfeld1998projective}}] \label{th:projectivetriangleminimal}
Any projective triangle is a minimal non-trivial blocking set of $\PG(\F_{q}^{3})$ having size $\frac{3(q+1)}{2}$.
\end{proposition}

A construction of such a point set is the following. Denote by $Q$ the set of non zero square elements of $\fq$, i.e.
\[
Q:=(\fq^*)^2=\{a^2 \colon a \in \fq^*\}.
\]
Recall that if $q$ is odd then $\lvert Q \rvert=\frac{q-1}{2}$.  
\begin{proposition} [see \textnormal{\cite[Lemma 13.6 (i)]{hirschfeld1998projective}}] \label{th:constructionprojectivetriangle}
Let $q$ be an odd prime. The point set
\begin{small}
\[
\mathcal{S}=\{\langle (0, 1, -s) \rangle_{\fq},\langle(-s, 0, 1)\rangle_{\fq},\langle (1, -s, 0) \rangle_{\fq} \colon s \in Q\} \cup \{\langle (1, 0, 0) \rangle_{\fq},\langle(0, 1, 0)\rangle_{\fq},\langle (0, 0, 1) \rangle_{\fq}\}
\]
\end{small} 
in $\PG(\F_{q}^{3})$ is a projective triangle. In particular, a projective triangle always exists.
\end{proposition}

In the next proposition, we show that is always possible to construct a minimal locally $q^{th}$ power set for any $q$ as in iii) of Theorem \ref{th:classification}. 
\begin{proposition} \label{th:constructionsecondminimalqthpower}
Let $Q$ be the set of non zero square elements of $\fq$. Let $p_1,p_2,p_3$ be three distinct primes. 
The set
\begin{small}
\[
B=\{p_2p_3^{q-s}, p_1^{q-s}p_3,p_1p_2^{q-s} \colon s \in Q\} \cup \{p_1,p_2,p_3\}
\]
\end{small} 
is a minimal $q^{th}$ power set having size $\frac{3(q+1)}{2}$.
\end{proposition}

\begin{proof} 
The point set
\begin{small}
\[
\mathcal{S}=\{\langle (0, 1, -s) \rangle_{\fq},\langle(-s, 0, 1)\rangle_{\fq},\langle (1, -s, 0) \rangle_{\fq} \colon s \in Q\} \cup \{\langle (1, 0, 0) \rangle_{\fq},\langle(0, 1, 0)\rangle_{\fq},\langle (0, 0, 1) \rangle_{\fq}\} \subseteq \PG(\F_{q}^{3}),
\]
\end{small} 
associated to $B$ is a projective triangle (see Proposition \ref{th:constructionprojectivetriangle}). Now, the assertion follows by Theorem \ref{th:classification} and Proposition \ref{th:projectivetriangleminimal}.
\end{proof}

\begin{example}
For $q=5$, the set of non-zero squares is $Q=\{1,4\}$. Hence, the set
\begin{small}
\[
B=\{p_1,p_2,p_3,p_2p_3^{4}, p_1^{4}p_3,p_1p_3,p_1p_2^{4},p_1p_2, p_{2}p_{3} \} 
\]
\end{small} 
is a minimal locally $5^{th}$ power set of cardinality $9$, according to Proposition \ref{th:constructionsecondminimalqthpower}. Here, $p_{1}, p_{2}$ and $p_{3}$ are three distinct primes. 
\end{example}

\section{$2$-dimensional Locally $q^{th}$ Power Set} \label{section:2dim}
Theorem \ref{th:classification} states that a minimal $2$-dimensional locally $q^{th}$ power set has cardinality at least $\frac{3(q+1)}{2}$, and in the case of equality the associated point set in $\PG(\F_{q}^{k})$ spans a planar blocking set. In this section, we will classify minimal $2$-dimensional locally $q^{th}$ power set for $q \in \{3,5,7\}$.

\subsection{For $q = 3$}
\begin{proposition} [see \textnormal{\cite[Theorem 13.21]{hirschfeld1998projective}}]
A non-trivial minimal blocking set of $\PG(\F_{3}^{3})$ has size $6$ and it is a projective triangle, that is unique up to the action of $\mathrm{PGL}(3,3)$.
\end{proposition}

As a consequence, we get that, up to the action of $\mathrm{PGL}(3,3)$, the only minimal blocking set of $\PG(2,3)$, is that described in Proposition \ref{th:constructionprojectivetriangle}.

\begin{proposition} \label{th:classq3}
    Let $q=3$ and $B$ be a minimal locally $3^{rd}$ power set and $\mathcal{S}$ be the point set in $\PG(\F_{q}^{k})$ associated to $B$. Then, $| B | = \frac{3(q+1)}{2} = 6$ if and only if $\mathcal{S}$ is geometrically $q$-equivalent to the set
    \[\{\overline{p}_1,\overline{p}_2,\overline{p}_3,\overline{p}_1\overline{p}_2^{2},\overline{p}_1\overline{p}_3^{2},\overline{p}_2\overline{p}_3^{2}\},\]
    for three distinct primes $\overline{p}_1,\overline{p}_2,\overline{p}_3$.
\end{proposition}

\subsection{For $q=5$}
We classify all the $q^{th}$ power set having the second minimum size $3(q+1)/2=9$. We start by recalling the following classification theorem for blocking sets of $\PG(\F_{5}^{3})$.
\begin{proposition} [see \textnormal{\cite[Theorem 13.25]{hirschfeld1998projective} and \cite[Section 3]{coolsaet2022classification}}]
A non-trivial minimal blocking set of $\PG(\F_{5}^{3})$ of minimum size $9$ is a projective triangle, that is unique up to the action of $\mathrm{PGL}(3,5)$.
\end{proposition}

As a consequence, we get that, up to the action of $\mathrm{PGL}(3,5)$, the only minimal blocking set of $\PG(\F_{5}^{3})$, is that described in Proposition \ref{th:constructionprojectivetriangle}.

\begin{proposition} \label{th:classq5}
Let $B$ be a minimal locally $5^{th}$ power set. Then, 
            $\lvert B \rvert=\frac{3(q+1)}{2}$ if and only if $B$ is geometrically $5$-equivalent to the set \[\{\overline{p}_1,\overline{p}_2,\overline{p}_3,\overline{p}_1\overline{p}_2,\overline{p}_1\overline{p}_3,\overline{p}_2\overline{p}_3,\overline{p}_1\overline{p}_2^{4},\overline{p}_1\overline{p}_3^{4},\overline{p}_2\overline{p}_3^{4}\},\]
    for three distinct primes $\overline{p}_1,\overline{p}_2,\overline{p}_3$.
\end{proposition}

\subsection{For $q=7$}
Unlike the case $q \in \{3,5\}$, when $q=7$, it is possible to find a non-trivial minimal blocking set $\mathcal{S}$ of $\PG(\F_{7}^{3})$, that is not $\mathrm{PGL}(3,7)$-equivalent to the projective triangle. This construction is related to the Hessian configuration of $\PG(\F_{7}^{3})$, that is, the set $\mathcal{C}$ of nine points
\[
\mathcal{C}= \left\{ \begin{array}{ccc}
     (1,-1, 0) \rangle_{\fq}, & \langle (0,1,-1)\rangle_{\fq}, &\langle (-1, 0, 1) \rangle_{\fq}  \\
\langle (1,-\omega, 0) \rangle_{\fq}, & \langle(0,1,-\omega) \rangle_{\fq}, & \langle (-\omega, 0, 1) \rangle_{\fq} \\
\langle (1,-\omega^2, 0) \rangle_{\fq}, &\langle(0,1,-\omega^2) \rangle_{\fq}, &\langle (-\omega^2, 0, 1) \rangle_{\fq} 
\end{array}\right\},
\]
where $\omega$ is a cubic root of unity in $\F_{7}$ (for instance, one can take $\omega=2$). We can find a blocking set of cardinality $\frac{3(7+1)}{2} = 12$, related to the Hessian, in the following way:

There are 12 trisecants $\{\ell_1, \ell_{2}, \ldots,\ell_{12}\}$ of $\mathcal{C}$, i.e. lines of $\PG(\F_7^3)$ meeting $\mathcal{C}$ in exactly three points. The set of points $\mathcal{S}=\{\ell_1^{\perp}, \ell_{2}^{\perp}, \ldots, \ell_{12}^{\perp}\}$ forms a non-trivial minimal blocking set of $\PG(2,7)$ having size $12$ that is not equivalent under the action of $\mathrm{PGL}(3,q)$ to a projective triangle, see \cite[Theorem 4]{coolsaet2022classification}. A representation of $\mathcal{S}$ is the following 
\begin{equation} \label{eq:trisecants}
\mathcal{S}= \left\{ \begin{array}{cccc}
     (1,0, 0) \rangle_{\fq}, & \langle (0,1,0)\rangle_{\fq}, &\langle (0, 0, 1) \rangle_{\fq}, & \langle (1, 1, 1) \rangle_{\fq}  \\
\langle (\omega^2,\omega, 1) \rangle_{\fq}, & \langle(\omega^2,1,\omega) \rangle_{\fq}, & \langle (\omega, \omega, 1) \rangle_{\fq}, & \langle (1, \omega, 1) \rangle_{\fq} \\
\langle (\omega, 1,1) \rangle_{\fq}, &\langle(\omega,1,\omega) \rangle_{\fq}, &\langle (1,\omega,\omega) \rangle_{\fq}, & \langle (1,1,\omega) \rangle_{\fq}  
\end{array}\right\}.
\end{equation}

\begin{proposition} [see \textnormal{\cite[Theorem 4]{coolsaet2022classification}}] \label{th:classq7}
 The point set $\mathcal{S}$ as in \eqref{eq:trisecants} is minimal blocking set of $\PG(\F_{7}^{3})$ of minimum size $12$ that is not $\mathrm{PGL}(3,7)$ equivalent to a projective triangle.
\end{proposition}

Consequently, we obtain a different behavior of minimal locally $7^{th}$ power set of the second smallest size $\lvert B \rvert=3(q+1)/2$, in contrast to the case when $q \in \{3,5\}$ (cf. \ref{th:classq3} and \ref{th:classq5}).

\begin{theorem} \label{th:upperboundqthpowerset}
Let $p_1,p_2,p_3$ be three distinct primes and $\omega$ be a cubic root of the unity in $\F_7$. Then, the set \[B_{1} = \{p_1,p_2,p_3,p_1p_2^3,p_1p_3^3,p_2p_3^3,p_1p_2^{5},p_1p_3^{5},p_2p_3^{5},p_1p_2^{6},p_1p_3^{6},p_2p_3^{6}\},\]
and the set
\[B_{2} = \{p_1,p_2,p_3,p_1p_2p_3, p_1^{\omega^2}p_2^{\omega}p_3, p_1^{\omega^2}p_2p_3^{\omega}, p_1^{\omega}p_2^{\omega}p_3,p_1p_2^{\omega}p_3,p_1^{\omega}p_2p_3, p_1^{\omega}p_2p_3^{\omega},p_1p_2^{\omega}p_3^{\omega},p_1p_2p_3^{\omega}\}\]
are locally $7^{th}$ power sets of minimum size $3(q+1)/2=12$, that are not geometrically $7$-equivalent.  
\end{theorem}

\section{$3$-dimensional locally $q^{th}$ power set}\label{section:3dim}

This last section is to provide insights into $3$-dimensional locally $q^{th}$ power sets for $q = 3, 5$.
\begin{proposition}
Let $q \in\{3, 5\}$ and $\mathcal{S}$ be a $3$-dimensional minimal blocking set of $\PG(\F_{q}^{4})$. Then, $\lvert \mathcal{S} \rvert \geq 2q+1$.
\end{proposition}

The above bound is sharp. Constructions of point sets having size meeting the above bound were provided by Tallini in \cite{tallini1991blocking}. 
\begin{definition}
If $\mathcal{T}$ is a set of $q + 1$ points, then a point $\mathcal{P} \notin \mathcal{T}$ is called a \emph{nucleus} of $\mathcal{T}$ if every line through $\mathcal{P}$
contains exactly one point of $\mathcal{T}$.
\end{definition}

\begin{proposition} [see \textnormal{\cite[Example 1]{tallini1991blocking}}] \label{prop:blockinglength3}
Let $\ell$ be a line in $\PG(\F_{q}^{4})$, $\pi$ be a plane in $\PG(\F_{q}^{4})$ not containing $\ell$, and $T$ be a set of $(q+1)$ non-collinear points of $\pi$ having the point $\pi \cap \ell$ as nucleus. Then, the point set 
\[
\mathcal{S}=(\ell \setminus \pi) \cup \mathcal{T} \subseteq \PG(\F_{q}^{4}),
\]
is a 3-dimensional minimal blocking set having size $2q+1$.
\end{proposition}

Using the above results, we can construct explicit examples of minimal 3-dimensional locally $q^{th}$-power set for $q \in \{3,5\}$ of size $2q+1$. Choose a coordinatization for $\PG(\F_{q}^{4})$ and let $\pi$ be the plane of $\mathrm{PG}(\F_q^4)$ having equation $X_3=0$. Consider the set 
\[
\mathcal{T}=\{\langle (1,t,t^2,0) \rangle_{\fq} \colon t \in \F_q\} \cup \{ \langle 0,0,1 \rangle\} \subseteq \pi.
\]
It is easy to see that the point set $\mathcal{T}$ has as nucleus the point $\mathcal{Q}=\langle (0,1,0,0) \rangle_{\F_q}$ and consider $\ell$ be the line having equation $X_1=X_3=0$. We have that $\{\mathcal{Q}\}=\ell \cap \pi$. Then $\pi,\ell$ and $\mathcal{T}$ are as in Theorem \ref{prop:blockinglength3}. 
This means that the point set
\[
\mathcal{S}=\{\langle (1,t,t^2,0) \rangle_{\fq} \colon t \in \F_q\} \cup \{\langle (0,1,0,0) \rangle_{\fq},\langle (0,0,1,0) \rangle_{\fq}\} \cup \{\langle  (0,1,0,t)\rangle_{\fq} \colon t \in \F_q^*\}
\]
is a $3$-dimensional minimal blocking set having size $2q+1$. Considering the set of integers associated to $\mathcal{S}$, we have the following theorem.

\begin{theorem}
Let $B$ be a minimal $3$-dimensional locally $q^{th}$ power set. Then $\lvert B \rvert \geq 2q+1$. Moreover, the sets 
\[
B_1=\{p_1,p_2,p_3,p_2p_3,p_1p_2^2p_3,p_2p_4,p_2p_4^2\}
\] and
\[
B_2=\{p_1,p_2,p_3,p_2p_3,p_1p_2^2p_3^4,p_1p_2^3p_3^4,p_1p_2^4p_3,p_2p_4,p_2p_4^2,p_2p_4^3,p_2p_4^4\}
\]
are minimal, $3$-dimensional, locally $3^{rd}$ power set and minimal, $3$-dimensional, locally $5^{th}$ power set, respectively.
\end{theorem}

\subsection{Upper bound on minimal locally $q^{th}$ power set}
In contrast to the lower bounds on cardinality of minimal locally $q^{th}$ power sets, one also has upper bounds on the same cardinality.

\begin{proposition}
Let $\mathcal{S}$ be a minimal blocking set of $\PG(\F_{q}^{k})$ and $s$ denote the fractional part of $\sqrt{q}$. Then:
\begin{enumerate}
        \item if $k=3$, then $\lvert \mathcal{S}   \rvert \leq \begin{cases} q \sqrt{q}+1 - \frac{s(1-s)q}{4} \text{ ; for } q \neq 5 \\ q \sqrt{q}+1 \text{ ; for } q = 5. \end{cases}$ \vspace{2mm}
        
        \item if $k=4$, then $\lvert \mathcal{S} \rvert \leq q^2+1$. \vspace{2mm}
        
        \item if $k\geq 5$, then $\lvert \mathcal{S} \rvert < \sqrt{q^{k}}+1$.
\end{enumerate}
\end{proposition}
The result above for $k = 3$ and $q \neq 5$ has been first shown in \cite{szHonyi2005large} and the rest appears in \cite[Theorem 1]{bruen1982hyperplane}. As a consequence we obtain the following.
\begin{corollary}
Let $B$ be a finite subset of integers not containing a perfect $q^{th}$ power that is a minimal locally $q^{th}$ power set. Let $s$ denote the fractional part of $\sqrt{q}$. Then:
\begin{enumerate}
        \item if $k=3$, then $\lvert B \rvert \leq \begin{cases} q \sqrt{q}+1 - \frac{s(1-s)q}{4} \text{ ; for } q \neq 5 \\ q \sqrt{q}+1 \text{ ; for } q = 5. \end{cases}$ \vspace{2mm}
        
        \item if $k=4$, then $\lvert B \rvert \leq q^2+1$. \vspace{2mm}
        
        \item if $k\geq 5$, then $\lvert B \rvert < \sqrt{q^{k}}+1$.
\end{enumerate}
\end{corollary}

When $k=4$, 3-dimensional minimal blocking sets of $\PG(\F_{q}^{4})$ having size $q^2+1$ are classified and they correspond to an ovoid \cite{bruen1982hyperplane}, which correspond to an elliptic quadric of $\PG(\F_{q}^{4})$ \cite{barlotti1955estensione,panella1955caratterizzazione}. These results are summarized in the following proposition. 

\begin{proposition}
A 3-dimensional minimal blocking set of $\PG(\F_{q}^4)$ having maximum size $q^2+1$, with $q$ odd, is an elliptic quadric. 
\end{proposition}

Note also that the elliptic quadric of $\PG(\F_{q}^{4})$ are classified, see e.g. \cite[p. 123]{hirschfeld1998projective}.

\begin{proposition} 
    An elliptic quadric in $\PG(\F_{q}^{4})$, up to $\mathrm{PGL}(4,q)$-equivalence, has equation 
    \[
    f(X_1,X_2)+X_2X_3=0,
    \]
    where $f(X_{1}, X_{2}) \in \F_q[X_{1}, X_{2}]$ is an irreducible polynomial.
\end{proposition}

For instance, let $f(X_{1}, X_{2}) = X_{1}^2-\alpha X_{2}^2$, with $\alpha \notin Q$. As a consequence, we have the following:

\begin{theorem}
$B$ be a minimal, $3$-dimensional, locally $q^{th}$ power set. Then, $\lvert B \rvert \leq q^2+1$. Moreover, $\lvert B \rvert=q^2+1$ if and only if $B$ is geometrically $q$-equivalent to the set
\[
\left\{ p_1^{v_{11}}p_2^{v_{21}}p_3^{v_{31}}p_4^{v_{41}}, \ldots,  p_1^{v_{1,q^2+1}}p_2^{v_{2,q^2+1}}p_3^{v_{3,q^2+1}}p_4^{v_{4,q^2+1}}\right\},
\]
where
$\{\langle (v_{1j},v_{2j},v_{3j},v_{4j}) \rangle_{\F_q}\big\}_{j=1}^{q^{2}+1}\subset\PG(\F_{q}^{4})$ is the set of $q^{2}+1$ points on an elliptic quadric.
\end{theorem}

\section*{Acknowledgement}\footnotesize
The authors are very grateful to the reviewer for their invaluable comments and suggestions, which have significantly improved the exposition of this article. The second author research was partially supported by the Italian National Group for Algebraic and Geometric Structures and their Applications (GNSAGA - INdAM) and by the INdAM - GNSAGA Project \emph{Tensors over finite fields and their applications}, number E53C23001670001 and by Bando Galileo 2024 – G24-216.

\bibliographystyle{abbrv}
\bibliography{blocking}

\medskip

Bhawesh Mishra \\
Department of Mathematical Sciences, \\
384 Dunn Hall, University of Memphis,\\
Memphis, TN 38107, USA \\
{{\em bmishra1@memphis.edu}}

\medskip

Paolo Santonastaso\\
Dipartimento di Matematica e Fisica,\\
Universit\`a degli Studi della Campania ``Luigi Vanvitelli'',\\
I--\,81100 Caserta, Italy\\
{{\em paolo.santonastaso@unicampania.it}}

\end{document}